\documentclass[12pt]{amsart}

\usepackage{fullpage}
\usepackage{amsmath}
\usepackage{amssymb}
\usepackage{amsthm}
\usepackage{amstext}
\usepackage{appendix}
\usepackage{perpage}
\usepackage{color}
\usepackage{tikz-cd}
\usepackage{slashed}
\usepackage[colorlinks=true,linkcolor=blue, citecolor=blue]{hyperref}

\newtheorem{thm}{Theorem}[section]
\newtheorem{lemm}[thm]{Lemma}
\newtheorem{coro}[thm]{Corollary}
\newtheorem{prop}[thm]{Proposition}

\theoremstyle{definition}

\newtheorem{defi}[thm]{Definition}
\newtheorem{remark}[thm]{Remark}
\newtheorem{quest}[thm]{Question}

\begin{document}

\title{Twisting operators and centralisers of Lie type groups over local rings}

\author{Zhe Chen}

\address{Department of Mathematics, Shantou University, Shantou, China}

\email{zhechencz@gmail.com}

\begin{abstract}
We extend the classical result asserting that the twisting operator preserves certain Deligne--Lusztig character values for truncated formal power series; along the way we discuss some properties of centralisers. This leads us to the construction of an action of $\mathrm{GL}_n(\mathbb{F}_q[[\pi]]/\pi^r)$ on a Springer fibre intersected by Deligne--Lusztig varieties; we determine the primitivities of the induced cohomological representations for single cycles. The case of $\mathrm{SL}_2$ over finite dual numbers is presented with a criterion on semisimple orbit representations.
\end{abstract}

\maketitle

\tableofcontents

\section{Introduction}\label{section:Intro}

Let $\mathbb{F}_q$ be a finite field with $q$ elements. Consider the finite groups $\mathrm{GL}_n(\mathbb{F}_{q^d})$ and $\mathrm{GL}_n(\mathbb{F}_q)$, where $d\in\mathbb{Z}_{>0}$. A natural question to wonder is: What are the relations between their characters? An answer was given by Shintani \cite{Shintani_twoRemarks} in 1976: He established a bijection from the set of Galois-invariant irreducible characters of $\mathrm{GL}_n(\mathbb{F}_{q^d})$ to the set of irreducible characters of $\mathrm{GL}_n(\mathbb{F}_q)$. This work was later developed into the theory of Shintani descents, which plays an important role in the geometric side of character theory; see \cite{Kawanaka_IrrCharUnitaryGp}, \cite{Lusztig_84_OrangeBook}, \cite{Asai_twisting_op_classicalgroup}, \cite{DM_L-function_Book}, \cite{DM_LusztigFunctorShintaniDescent}, \cite{Shoji_ShintaniAlgGp}, \cite{Shoji_ShintaniSL(n)}, \cite{Deshpande_ShintaniDescent_compositio}.

\vspace{2mm} In the very special case $d=1$, Shintani descent is the natural operation switching the Lang isogeny on the space of class functions
$$f\longmapsto\mathrm{Sh}_G(f)\colon h^{-1}F(h)\mapsto f(F(h) h^{-1}),$$ 
called the twisting operator (see Definition~\ref{defi:Sh}), whose behaviour can be mysterious for groups admitting disconnected centralisers. In this paper we study this operator and related centraliser properties for connected reductive groups over a (truncated) discrete valuation ring in positive characteristic (like $\mathrm{SL}_n(\mathbb{F}_q[[\pi]]/\pi^r)$).

\vspace{2mm} In Section~\ref{sec:pre} we recall some basic properties of $\mathrm{Sh}_G$. And in Section~\ref{sec:Group over local ring} we make some preparations on reductive groups over formal power series rings and recall the orbit construction. 

\vspace{2mm} In Section~\ref{sec:higher DL} we consider two aspects concerning Deligne--Lusztig theory. First, by using a generalised Green function character formula and by generalising a basic property of semisimple centralisers, we extend a result of Asai and Digne--Michel, showing that $\mathrm{Sh}_G$ preserves the higher Deligne--Lusztig character values under a condition on unipotent parts; see Theorem~\ref{thm:main result}. 

\vspace{2mm} Second, motivated by a group embedding used in a centraliser discussion (in Proposition~\ref{prop:unipotent centraliser prop}), we observe that $\mathrm{GL}_n(\mathbb{F}_q[[\pi]]/\pi^r)$ can be viewed as a unipotent centraliser; this allows us to construct an action of $\mathrm{GL}_n(\mathbb{F}_q[[\pi]]/\pi^r)$ on the intersection $\mathcal{B}_{u,w}$ of a Springer fibre and a Deligne--Lusztig variety (see Definition~\ref{defi:B_{u,w}}). While both Springer fibres and Deligne--Lusztig varieties are of vital importance in geometric representation theory, it seems these intersections $\mathcal{B}_{u,w}$ have not been considered in literature before. We determine the primitivities of the representations afforded by the cohomology of these varieties for the cycles $w=(1,2,...,z)$; see Theorem~\ref{thm:R_{u,w}}.

\vspace{2mm} In Section~\ref{sec:expl} we study the behaviour of the twisting operator on primitive representations of $\mathrm{SL}_2([[\pi]]/\pi^2)$ in odd characteristic. We find that the semisimple orbit representations are exactly the ones invariant under $\mathrm{Sh}_G$; see Corollary~\ref{coro:semisimple of SL_2 are twist-invariant}.

\vspace{2mm}{\bf Acknowledgement.} The author thanks Alexander Stasinski for several useful comments on an earlier version of this paper, and thanks Yongqi Feng for helpful conversations.

\section{Preliminaries}\label{sec:pre}

In this section we recall some basics of the twisting operator following \cite{DM_L-function_Book}. 

\vspace{2mm} Let $G$ be the $\overline{\mathbb{F}}_q$-base change of a connected smooth affine group scheme over ${\mathbb{F}}_q$, and let $F$ be the associated geometric Frobenius endomorphism. Denote by $L\colon x\mapsto x^{-1}F(x)$ the Lang isogeny. The Lang--Steinberg theorem tells us that, for any $g\in G$, there is an $h\in G$ such that $h^{-1}F(h)=g$. Note that, if $g\in G^F$, then $F(h)h^{-1}\in G^F$; via this property one can define a permutation on the conjugacy classes of $G^F$:

\begin{defi}\label{defi:n_F}
Let $g=h^{-1}F(h)\in G^F$ and let $[g]_{G^F}$ be the conjugacy class in $G^F$. Then
$$n_F([g]_{G^F}):=[F(h)h^{-1}]_{G^F}.$$
(Note that this is independent of the choice of $h$.)
\end{defi}

\begin{defi}\label{defi:Sh}
The twisting operator $\mathrm{Sh}_G$ is the linear automorphism on the complex space of class functions of $G^F$ (denoted by $\mathcal{C}(G^F)$) given by composition with $n_F$, that is:
$$\mathrm{Sh}_G(f)=f\circ n_F$$
for $f\in \mathcal{C}(G^F)$.
\end{defi}

\begin{lemm}
The operator $\mathrm{Sh}_G\in \mathrm{End}(\mathcal{C}(G^F))$ is semisimple with eigenvalues being roots of unity. Moreover, $\mathrm{Sh}_G$ preserves the usual inner product on $\mathcal{C}(G^F)$.
\end{lemm}
\begin{proof}
By definition, $\mathrm{Sh}_G$ is a permutation matrix with respect to the basis consisting of characteristic functions on conjugacy classes, from which the first two properties follow. Now the property of being isometric follows from the fact that (see \cite[I.7.2]{DM_L-function_Book}) $n_F$ preserves the cardinals of conjugacy classes.
\end{proof}

A fundamental property of $\mathrm{Sh}_G$ is the following lemma:

\begin{lemm}\label{lemm:a fundamental lemma}
Let $g\in G^F$. Consider the set of $F$-rational elements in the geometric conjugacy class (i.e.\ the $G$-conjugacy class) of $g$ in $G$; partition this set into $G^F$-conjugacy classes. There is a natural bijection between the resulting $G^F$-classes and the $F$-conjugacy classes of the component group $Z_G(g)/Z_G(g)^{\circ}$, such that, under this bijection the action of $n_{F}$ becomes the multiplication by $g$.
\end{lemm}
\begin{proof}
See \cite[IV.1.1]{DM_L-function_Book} or \cite[3.21]{DM1991}.
\end{proof}

From the above it follows immediately that:

\begin{coro}
The order of $\mathrm{Sh}_G$ is the smallest $i$ such that, for any $g\in G^F$, $g^i=1$ in all the $F$-conjugacy classes of $Z_G(g)/Z_G(g)^{\circ}$.
\end{coro}

\begin{lemm}\label{lemm:triviality for commutative groups}
If $G$ is commutative, then $n_F$, and hence $\mathrm{Sh}_G$, are the identity maps.
\end{lemm}
\begin{proof}
This follows from the definition.
\end{proof}

\begin{lemm}\label{lemm:triviality on ss classes}
The permutation $n_{F}$ is the identity on the semisimple classes.
\end{lemm}
\begin{proof}
Let $s\in G^F$ be semisimple, then $s$ is contained in an $F$-stable maximal torus $S$ of $G$ (see e.g.\  \cite[3.16]{DM1991}). The assertion follows from Lemma~\ref{lemm:triviality for commutative groups} by restricting $n_F$ to $S^F$.
\end{proof}

\section{Groups over local rings}\label{sec:Group over local ring}

Let $\mathcal{O}:=\mathbb{F}_q[[\pi]]$, and let $\mathbb{G}$ be a connected reductive group scheme over $\mathcal{O}$. Note that every smooth representation of $\mathbb{G}(\mathcal{O})$ factors through $\mathbb{G}(\mathcal{O}_r)$ for some $r\in \mathbb{Z}_{>0}$, where $\mathcal{O}_r:=\mathbb{F}_q[[\pi]]/\pi^r$. From now on we fix an arbitrary such $r$.

\vspace{2mm} We briefly recall the settings in \cite{Lusztig2004RepsFinRings}: There is a connected algebraic group $G_r$ (Weil restriction), defined over $\mathbb{F}_q$, such that 
$$G_r(\overline{\mathbb{F}}_q)\cong \mathbb{G}(\overline{\mathbb{F}}_q[[\pi]]/\pi^r) \ \ \text{and}\ \ G_r^F:=G_r(\overline{\mathbb{F}}_q)^F\cong\mathbb{G}(\mathcal{O}_r),$$
where $F$ is the associated geometric Frobenius endomorphism. For $i\in\mathbb{Z}_{[1,r-1]}$, the reduction map modulo $\pi^i$ induces a surjective morphism between algebraic groups 
$$\rho_{r,i}\colon G_r\longrightarrow G_i;$$
we denoted the kernel group (congruence group) by $G^i$. Note that $G_1$ is a connected reductive group over $\overline{\mathbb{F}}_q$ and we have a semi-direct product 
$$G_r\cong G_1\ltimes G^1,$$
which is a Levi decomposition. Every closed subgroup scheme of $\mathbb{G}\times_{\mathcal{O}}\overline{\mathbb{F}}_q[[\pi]]/\pi^r$ corresponds to a closed subgroup of $G_r$; for such a subgroup $H\subseteq G_r$, we shall use similar notations $H_i$ and $H^i$. As before, $L\colon g\mapsto g^{-1}F(g)$ denotes the Lang isogeny associated to $F$.  In an algebraic group, we use the conjugation notation $x^y:=y^{-1}xy=:{^{y^{-1}}x}$.

\vspace{2mm} Throughout this paper we would write $G:=G_r$ when there is no confusion.

\vspace{2mm}  An important tool for studying the representations of $G^F=\mathbb{G}(\mathcal{O}_r)$, when $r\geq2$, is the notion of orbits. In the remaining part of this section we assume that $r\geq 2$, and that the Lie algebra $\mathfrak{g}:=\mathrm{Lie}(G_1)$ admits a non-degenerate $G_1$-equivariant bilinear form $\mu\colon \mathfrak{g}\times\mathfrak{g}\rightarrow\overline{\mathbb{F}}_q$ defined over $\mathbb{F}_q$. (This requires $\mathrm{char}(\mathbb{F}_q)$ is not too small unless $\mathbb{G}=\mathrm{GL}_n$; see e.g.\ the discussions in \cite[2]{Let2005book}.)  We fix such a $\mu$ and fix a non-trivial character $\phi\colon \mathbb{F}_q\rightarrow \overline{\mathbb{Q}}_{\ell}^{\times}$. Note that $G^{r-1}$ can be viewed as the additive group of $\mathfrak{g}$; under this viewpoint, the conjugation action of $G$ on $G^{r-1}$ is translated to be the adjoint action of $G_1$ on $\mathfrak{g}$. If $\sigma$ is an irreducible character of $G^F$, then by Clifford's theorem (see \cite[6.2]{Isaacs_CharThy_Book}) we have
$$\sigma|_{(G^{r-1})^F}= e\cdot\left(\sum_{\chi\in\Omega(\sigma)} \chi \right),$$
where $e$ is a positive integer and $\Omega(\sigma)$ is a $G_1^F$-orbit of irreducible characters of the additive group $\mathfrak{g}^F$, so we get a map $\Omega\colon \mathrm{Irr}(G^F)\rightarrow \mathbb{G}(\mathbb{F}_q)\backslash\mathrm{Irr}(\mathfrak{g}^F)$. Taking the dual of $\Omega$ via $\phi(\mu(-,-))$ we get another map:
$$\Omega'\colon \mathrm{Irr}(G^F)\longrightarrow \mathbb{G}(\mathbb{F}_q)\backslash \mathfrak{g}^F.$$
We call $\Omega'(\sigma)$ the orbit of $\sigma$; of particular interest is the case that $\Omega'(\sigma)$ is nilpotent, regular, or semisimple. 

\vspace{2mm} Constructing irreducible characters of $G^F$ using orbits attracts many attentions; see e.g.\ the recent works \cite{Stasinski_Stevens_2016_regularRep} and \cite{Krakovski_Onn_Singla_regularchar_2018}, and the references therein. Note that the notion of orbits naturally extends to the whole space $\mathcal{C}(G^F)$: If $f=\oplus_i{\lambda_i}\sigma_i$ is a class function with irreducible components $\sigma_i$, then the formal sum $\Omega'(f):=\sum_i\lambda_i\Omega'(\sigma_i)$ gives a surjective linear map
$$\Omega'\colon \mathcal{C}(G^F)\longrightarrow \mathbb{C}[\mathbb{G}(\mathbb{F}_q)\backslash \mathfrak{g}^F],$$
where $\mathbb{C}[\mathbb{G}(\mathbb{F}_q)\backslash \mathfrak{g}^F]$ denotes the space of formal sums of adjoint orbits.

\begin{prop}\label{prop:Sh preserves orbits}
Let $f$ be a class function of $G_r^F$, then $\Omega'(f)=\Omega'(\mathrm{Sh}_G(f))$.
\end{prop}
\begin{proof}
By construction, $\Omega'(f)$ only depends on the restriction of $f$ to $(G^{r-1})^F$; meanwhile, we have $\mathrm{Sh}_G(f)\mid_{(G^{r-1})^F}=f\mid_{(G^{r-1})^F}$ by Lemma~\ref{lemm:triviality for commutative groups}.
\end{proof}

\section{Deligne--Lusztig constructions}\label{sec:higher DL}

In \cite{Lusztig1979SomeRemarks} and \cite{Lusztig2004RepsFinRings} Lusztig gives the Deligne--Lusztig construction for $\mathbb{G}(\mathcal{O}_r)$. Recently, this work has been studied and generalised along various aspects; see e.g.\ \cite{Sta2009Unramified}, \cite{Sta2011ExtendedDL}, \cite{Chen_2017_InnerProduct}, \cite{Chen_2019_flag_orbit}, \cite{ChanetIvanov_CohomRepParahoricGp}. In this section we study the representations of $\mathbb{G}(\mathcal{O}_r)$ via Deligne--Lusztig theory in two directions, one through the twisting operators and one through the Springer fibres.

\vspace{2mm} For $r=1$, the relations between twisting operators and Deligne--Lusztig theory have been investigated extensively. Here we give a start for general $r\geq 1$. We will work in the setting of \cite{Chen_2017_InnerProduct}, and try to extend the classical result asserting that Deligne--Lusztig character values are invariant under the twisting operator; when $r=1$, this was obtained by Asai \cite{Asai_twisting_op_classicalgroup} for classical groups with connected centre and by Digne--Michel \cite{DM_L-function_Book} for Lusztig inductions in good characteristic. To properly state the result we also need to extend a basic property of semisimple centralisers.

\vspace{2mm} Let $\mathbf{B}=\mathbf{T}\ltimes\mathbf{U}$ be a Borel subgroup of $\mathbb{G}\times_{\mathcal{O}}\overline{\mathbb{F}}_q[[\pi]]/\pi^r$, where $\mathbf{T}$ is a maximal torus and $\mathbf{U}$ is the unipotent radical. Let $B_r, U_r, T_r\subseteq G_r$ be the corresponding algebraic groups; similar to $G=G_r$, we may omit the subscript $r$ and denote them by $B,U,T$. We only concern the case that $T$ is $F$-stable (i.e.\ $FT=T$). Let $\ell$ be a prime not equal to $\mathrm{char}(\mathbb{F}_q)$, and fix a field isomorphism $\mathbb{C}\cong \overline{\mathbb{Q}}_{\ell}$; in particular, we can view $\mathrm{Sh}_G$ as an operator on the $\overline{\mathbb{Q}}_{\ell}$-space of class functions.

\vspace{2mm} For $b\in\mathbb{Z}_{[0,r]}$, write $U^{b,r-b}:=U^b(U^-)^{r-b}$, where $U^-$ denotes the algebraic group corresponding to the opposite of $\mathbf{U}$. Note that $G^F$ (resp.\ $T^F$) acts on $L^{-1}(FU^{b,r-b})$ by left (resp.\ right) translation. 
\begin{defi}
For $\theta\in \mathrm{Hom}(T^F,\overline{\mathbb{Q}}_{\ell})$, the corresponding higher Deligne--Lusztig representation of $G^F$ at page $b$ is the virtual representation
$$R_{T,U,b}^{\theta}:=\sum_i(-1)^i H^i_c (L^{-1}(FU^{b,r-b}),\overline{\mathbb{Q}}_{\ell})_{\theta},$$
where $H_c^i(-,\overline{\mathbb{Q}}_{\ell})_{\theta}$ denotes the $\theta$-isotypical part of the compactly supported $i$-th $\ell$-adic cohomology group.
\end{defi}

\begin{remark}
When $b=0$ or $r$, $R_{T,U,b}^{\theta}$ is the representation constructed by Lusztig in \cite{Lusztig2004RepsFinRings}. When $r$ is even and $b=r/2$, $R_{T,U,b}^{\theta}$ is isomorphic to the representation constructed by G\'erardin in \cite{Gerardin1975SeriesDiscretes}. When $B$ is $F$-stable, $R_{T,U,b}^{\theta}\cong \mathrm{Ind}_{T^F(U^{b,r-b})^F}^{G^F}\theta$.
\end{remark}

Let $\mathfrak{R}_{T,U,b}^{\theta}$ be the character of $R_{T,U,b}^{\theta}$.

\begin{thm}\label{thm:main result}
Let $g=su\in G^F$, where $s$ is the semisimple part and $u$ is the unipotent part. Then $u\in Z_{G}(s)^{\circ}$, and if moreover $u\in Z_{Z_{G}(s)^{\circ}}(u)^{\circ}$, then
$$\mathrm{Sh}_G (\mathfrak{R}_{T,U,b}^{\theta})(g)=\mathfrak{R}_{T,U,b}^{\theta} (g).$$
\end{thm}
\begin{proof}
The first part is essentially proved in the arguments in \cite[P2031-2032]{Chen_2018_GreenFunction}; here we include the details for the completeness. Let $c\in G$ be such that $s^c\in T_1\subseteq G_1$, then it is sufficient to show that $u^c\in Z_{G}(s^c)^{\circ}$. 

\vspace{2mm} Write $u^c=u_1u_2$ with respect to the decomposition $G=G_1G^1$. By applying the reduction $\rho_{r,1}$ to the identity
$$u_1u_2\cdot s^c=s^c\cdot u_1u_2$$ 
we get $u_1\in Z_{G_1}(s^c)$, which in turn implies that $u_2\in Z_{G^1}(s^c)$. As $u_1\in Z_{G_1}(s^c)^{\circ}$ (see e.g.\ \cite[2.5]{DM1991}), it is then sufficient to show that $u_2\in Z_{G^1}(s^c)^{\circ}$. Actually we would show that $Z_{G^1}(s^c)=Z_{G^1}(s^c)^{\circ}$.

\vspace{2mm} Write $u_2=tv$ with $t\in T^1$ and $v\in U^1(U^-)^1$; this is the Iwahori decomposition in the sense of \cite[2.2]{Sta2009Unramified}. As $t,s^c\in T$ commute with each other, we see that $v$ also commutes with $s^c$. Now consider the decomposition $v=\prod_{\alpha}v_{\alpha}$ with $v_{\alpha}\in (U_{\alpha})^1$, where $\alpha$ runs over the roots with respect to $T_1$. Then the commutation condition 
$$(s^c)^{-1}\cdot v\cdot s^c=v$$ 
means that, for each $\alpha$, either $v_{\alpha}=1$ or $\alpha(s^c)=1$. Thus $Z_{G^1}(s^c)$ is an affine space, hence connected as desired.

\vspace{2mm} Now we turn to the second part. We shall start with a generalised Green function character formula, and then use a counting method following \cite[IV]{DM_L-function_Book}. 

\vspace{2mm} By \cite[3.3]{Chen_2018_GreenFunction} we have:
\begin{equation}\label{formula:induction character formula}
\begin{split}
\mathfrak{R}_{T,U,b}^{\theta} (g)=&\frac{1}{|(Z_{G}(s)^{\circ})^F|} \\
& \cdot  \sum_{   \left\lbrace\substack{   h\in G^F\ \text{s.t.} \\  s\in {^h(T_1)^F}  } \right\rbrace   } \sum_{\tau\in {^h(T^1)^F} }         {\theta}(s^h\cdot \tau^h)\cdot   Q_{{^hT},{^hU}\cap Z_G(s)^{\circ},b}^{Z_{G}(s)^{\circ}}(u,\tau^{-1}),
\end{split}
\end{equation}
where 
$$Q_{{^hT},{^hU}\cap Z_G(s)^{\circ},b}^{Z_{G}(s)^{\circ}}(-,-)\colon (\mathcal{U}_{Z_G(s)^{\circ}})^F\times {^h(T^1)^F}\longrightarrow\overline{\mathbb{Q}}_{\ell}$$ 
is the two-variable Green function in \cite[3.1]{Chen_2018_GreenFunction} (here $\mathcal{U}_{Z_G(s)^{\circ}}$ denotes the unipotent variety of $Z_G(s)^{\circ}$), which is defined by:
\begin{equation*}
(x,y)\longmapsto  \frac{1}{|T^F|}\mathrm{Tr}\left((x,y)  \mid \sum_{i}(-1)^i\cdot H_c^i(Z_G(s)^{\circ}\cap L^{-1}( {^h FU^{r-b,b} }),\overline{\mathbb{Q}}_{\ell})  \right).
\end{equation*}

\vspace{2mm} Since $s$ is contained in an $F$-stable maximal torus (see e.g.\ \cite[3.16]{DM1991}), by Lang--Steinberg theorem and the above first part of the theorem, there are $u'\in Z_{G}(s)^{\circ}$ and semisimple $s'\in Z_{G}(s)^{\circ}$ such that $u'^{-1}F(u')=u$ and $s'^{-1}F(s')=s$. Thus
\begin{equation*}
\begin{split}
n_{F}([g]_{G^F})= & n_{F}([s'^{-1}F(s')u'^{-1}F(u')]_{G^F})=n_{F}([u'^{-1}s'^{-1}F(s')F(u')]_{G^F})\\
= & [F(s')F(u')u'^{-1}s'^{-1}]_{G^F}\\
= & [F(s')s'^{-1}(F(u')u'^{-1})^{s'^{-1}}]_{G^F}\\
= & [s\cdot {^{s'}(F(u')u'^{-1})}]_{G^F}.
\end{split}
\end{equation*}
Note that $s\cdot {^{s'}(F(u')u'^{-1})}$ is a Jordan decomposition, so, evaluating \eqref{formula:induction character formula} with $n_F([g]_{G^F})$ we get
\begin{equation}\label{formula:induction character formula 2}
\begin{split}
& \mathrm{Sh}_G (\mathfrak{R}_{T,U,b}^{\theta})(g)\\
& =\frac{1}{|(Z_{G}(s)^{\circ})^F|} \cdot  \sum_{   \left\lbrace\substack{   h\in G^F\ \text{s.t.} \\  s\in {^h(T_1)^F}  } \right\rbrace   } \sum_{\tau\in {^h(T^1)^F} }         {\theta}(s^h\cdot \tau^h)\cdot   Q_{{^hT},{^hU}\cap Z_G(s)^{\circ},b}^{Z_{G}(s)^{\circ}}( {^{s'}(F(u')u'^{-1})},\tau^{-1}).
\end{split}
\end{equation}

\vspace{2mm} Comparing \eqref{formula:induction character formula} and \eqref{formula:induction character formula 2} it remains to show that: For a given $h\in G^F$ with $s\in {^h(T_1)}$, and $\tau \in {^h(T^1)^F}$, we have
\begin{equation}\label{formula:green}
Q_{{^hT},{^hU}\cap Z_G(s)^{\circ},b}^{Z_{G}(s)^{\circ}}(u,\tau^{-1})=Q_{{^hT},{^hU}\cap Z_G(s)^{\circ},b}^{Z_{G}(s)^{\circ}}({^{s'}(F(u')u'^{-1})},\tau^{-1}).
\end{equation}
Note that, by our assumption on $u$ (that is, $u\in Z_{Z_{G}(s)^{\circ}}(u)^{\circ}$) and Lemma~\ref{lemm:a fundamental lemma}, we have 
$$F(u')u'^{-1}={^zu}$$ 
for some $z\in (Z_G(s)^{\circ})^F$, so ${^{s'}(F(u')u'^{-1})}={^{s''}u}$ where $s'':=s'z$. Hence \eqref{formula:green} can be written as
\begin{equation}\label{formula:green 2}
Q_{{^hT},{^hU}\cap Z_G(s)^{\circ},b}^{Z_{G}(s)^{\circ}}(u,\tau^{-1})=Q_{{^hT},{^hU}\cap Z_G(s)^{\circ},b}^{Z_{G}(s)^{\circ}}({^{s''}u},\tau^{-1}).
\end{equation}
Moreover, by definition $Q_{{^hT},{^hU}\cap Z_G(s)^{\circ},b}^{Z_{G}(s)^{\circ}}(-,-)$ can be viewed as the restriction of a class function on $(Z_G(s)^{\circ})^F\times {^h(T^1)^F}$. In particular, the value $Q_{{^hT},{^hU}\cap Z_G(s)^{\circ},b}^{Z_{G}(s)^{\circ}}({^{s''}u},\tau^{-1})$ is independent of the choice of $s'\in Z_G(s)^{\circ}$, as any two such $s'$'s are different only up to a left translation by some element in $(Z_G(s)^{\circ})^F$. Therefore, as $s\in {^h(T_1)}$, in the below we can assume that $s''\in {^h(T_1)}$. (So we would have ${\tau}=\tau^{s''}$ as both $s'',\tau$ are in the commutative group ${^hT}$.)

\vspace{2mm} By basic properties of Lefschetz numbers (see \cite[1.2]{Lusztig_whiteBk}), showing \eqref{formula:green 2} is equivalent to showing:
\begin{equation}\label{formula:counting points}
\begin{split}
\# \{ x\in Z_G(s)^{\circ}\cap &  L^{-1}( {^h FU^{r-b,b} })  \mid   {^{s''}u}\cdot F^m(x)\cdot \tau^{-1} =x  \} \\
& =\# \{ x\in Z_G(s)^{\circ}\cap L^{-1}( {^h FU^{r-b,b} })  \mid   {u}\cdot F^m(x)\cdot \tau^{-1}=x  \} 
\end{split}
\end{equation}
for all $m\in \mathbb{Z}_{>0}$ with $F^m(U^{r-b,b})=U^{r-b,b}$.

\vspace{2mm} Taking the variable change $x\mapsto s''y$, the LHS of \eqref{formula:counting points} becomes 
$$\# \{ y\in Z_G(s)^{\circ} \mid   sy^{-1}F(y)\in {^hFU^{r-b,b}} ,  u\cdot s''^{-1}F^m(s'')\cdot F^m(y)\cdot \tau^{-1} =y  \}. $$
(Recall that $s''^{-1}F(s'')=s$ by construction.) Note that 
$$s''^{-1}F^m(s'')=s''^{-1}F(s'')\cdot F(s''^{-1})F^2(s'')\cdots F^{m-1}(s''^{-1})F^m(s'')=sF(s)\cdots F^{m-1}(s)=s^m.$$
So the LHS of \eqref{formula:counting points} equals to:
$$\# \{ y\in Z_G(s)^{\circ} \mid   sy^{-1}F(y)\in {^hFU^{r-b,b}} ,  s^mu F^m(y) \tau^{-1} =y  \}. $$
Similarly, via the variable change $x\mapsto ys''$ the RHS of \eqref{formula:counting points} equals to: (Recall that both $s''$ and $\tau$ are in the $F$-stable commutative group ${^hT}$)
\begin{equation*}
\begin{split}
\# \{ y\in Z_G(s)^{\circ} \mid  & s\cdot (y^{-1}F(y))^{F(s'')}\in {^hFU^{r-b,b}} ,  s^mu F^m(y) \tau^{-1} =y  \}\\
& =\# \{ y\in Z_G(s)^{\circ} \mid   s\cdot y^{-1}F(y)\in {^hFU^{r-b,b}} ,  s^mu F^m(y) \tau^{-1} =y  \},
\end{split}
\end{equation*}
where the equality follows from that $F(s'')\in {^hT}$ stabilises ${^hFU^{r-b,b}}$. Thus \eqref{formula:counting points} holds.
\end{proof}

In the statement of the above theorem, note that if $s=1$, then the condition $u\in Z_{Z_{G}(s)^{\circ}}(u)^{\circ}$ becomes $u\in Z_G(u)^{\circ}$, and by Lemma~\ref{lemm:a fundamental lemma} the twisting operator does not affect the character values at such unipotent elements for any character. So we want to ask:

\begin{quest}\label{quest:centraliser quest}
When a unipotent element of $G_r$ lies in the identity component of its centraliser?
\end{quest}

In the case $r=1$, this is well-known (to be always true) if $p$ is a good prime, which is a classical application of the Springer homeomorphism; see \cite[Proposition~5]{McNinch--Sommers_UnipCentGoodChar} and \cite[III.3.15]{Springer--Steinberg_1970_Conj}. For a general $r$ we list some situations in the below.

\begin{prop}\label{prop:unipotent centraliser prop}
Let $g=su\in G_r^F$ be as in Theorem~\ref{thm:main result}.
\begin{itemize}
\item[(1)] If $p$ is good for $G_1$ and $u\in G_1$, then $u\in Z_{G_r}(u)^{\circ}$.

\item[(2)] If $u\in G_r^{[\frac{r+1}{2}]}$, then $u\in Z_{G_r}(u)^{\circ}$. (Here $[-]$ denotes the Gauss floor function.)

\item[(3)] If (i) $g\in G_1$ and $p$ is good for $Z_{G_1}(s)^{\circ}$, or if (ii) $u\in G^{[\frac{r+1}{2}]}$, then the condition ``$u\in Z_{Z_{G}(s)^{\circ}}(u)^{\circ}$'' in Theorem~\ref{thm:main result} holds.

\item[(4)] If $\mathbb{G}=\mathrm{GL}_n$, then $Z_{G}(g)$ is always connected; in particular, $\mathrm{Sh}_G$ is the identity map for $\mathrm{GL}_n$.
\end{itemize}
\end{prop}
\begin{proof}
(1) By \cite[12]{McNinch--Sommers_UnipCentGoodChar} we know $u$ lies in $Z_{G_1}(u)^{\circ}$, which is a closed subgroup of $Z_{G}(u)^{\circ}$.

\vspace{2mm} (2) As $[\frac{r+1}{2}]\geq r/2$, the group $G^{[\frac{r+1}{2}]}$ is abelian, so $u\in G^{[\frac{r+1}{2}]}=Z_{G^{[\frac{r+1}{2}]}}(u)^{\circ}\subseteq Z_{G}(u)^{\circ}$.

\vspace{2mm} (3-(i)) When $g\in G_1$ and $p$ is good for $Z_{G_1}(s)^{\circ}$, we can apply \cite[12]{McNinch--Sommers_UnipCentGoodChar} to $Z_{G_1}(s)^{\circ}$, which asserts that $u$ lies in $Z_{Z_{G_1}(s)^{\circ}}(u)^{\circ}$, a closed subgroup of $Z_{Z_{G}(s)^{\circ}}(u)^{\circ}$. 

\vspace{2mm} (3-(ii)) Now assume $u\in G^{[\frac{r+1}{2}]}$. Let $c\in G$ be such that $s^c\in G_1$. Then $(Z_{G}(s^c)^{\circ})^{[\frac{r+1}{2}]}=Z_{G}(s^c)^{\circ}\cap G^{[\frac{r+1}{2}]}$ is a connected unipotent group (see \cite[2.2]{Chen_2018_GreenFunction}), which is abelian as $[\frac{r+1}{2}]\geq r/2$. Therefore (by the first part of Theorem~\ref{thm:main result}) 
$$u^c\in (Z_{G}(s^c)^{\circ})^{[\frac{r+1}{2}]}=Z_{(Z_{G}(s^c)^{\circ})^{[\frac{r+1}{2}]}}(u^c)^{\circ}\subseteq Z_{Z_{G}(s^c)^{\circ}}(u^c)^{\circ},$$
thus $u\in Z_{Z_{G}(s)^{\circ}}(u)^{\circ}$.

\vspace{2mm} (4) Write an element in $M_n(\overline{\mathbb{F}}_q[[\pi]]/\pi^r)$ as $A_0+A_1\pi+...+A_{r-1}\pi^{r-1}$, then as observed in \cite[5.3]{Chen_2019_flag_orbit}, there is a ring injection from $M_n(\overline{\mathbb{F}}_q[[\pi]]/\pi^r)$ to ${M}_{nr}(\overline{\mathbb{F}}_q)$:
\begin{equation}\label{formula:injection}
A_0+A_1\pi+...+A_{r-1}\pi^{r-1}\longmapsto
\begin{bmatrix}
	A_0     & 0       & 0   & ... & 0   \\
	A_1     & A_0     & 0   & ... & 0   \\
	...     & ...     & ... & ... & ... \\
	A_{r-2} & ...     & A_1 & A_0 & 0   \\
	A_{r-1} & A_{r-2} & ... & A_1 & A_0
\end{bmatrix}.
\end{equation}
Thus, for $g\in\mathrm{GL}_n(\overline{\mathbb{F}}[[\pi]]/\pi^r)$, the condition ``$gx=xg$'' can be viewed as a system of linear equations in the linear space $\overline{\mathbb{F}}_q^{nr}$. So $Z_G(g)$ is an open subset of an affine space, thus connected, and $\mathrm{Sh}_G$ is the identity map by Lemma~\ref{lemm:a fundamental lemma} in this case.
\end{proof}

Now we turn to the other direction. In the remaining part of this section we take $\mathbb{G}$ to be $\mathrm{GL}_n$ and assume that $n,r\geq2$. 

\vspace{2mm} The ring injection \eqref{formula:injection} in the above argument restricts to a group embedding from $G_r$ into $\widehat{G_r}:={\mathrm{GL}_{nr}}_{/\overline{\mathbb{F}}_q}$. This realises $G_r$ as a unipotent centraliser:

\begin{lemm}
The image of the above group embedding is the centraliser of the unipotent element
\begin{equation*}
u=\begin{bmatrix}
I_n & 0 &  0 & 0 & ... & 0\\
I_n & I_n & 0 & 0 & ... & 0\\
0 & I_n & I_n & 0 & ... & 0\\
... & ... & ...  & ... & ... & ... \\
0 & ... & 0  & I_n & I_n & 0 \\
0 & 0 & ... & 0 & I_n & I_n\\
\end{bmatrix},
\end{equation*}
whose Jordan normal form is $\mathrm{diag}(J_r,...,J_r)$, where $I_n$ denotes the identity matrix of size $n$ and $J_r$ denotes the unipotent Jordan block of size $r$. In particular, $G_r$ acts on the Springer fibre $\mathcal{B}_u$, the variety of Borel subgroups (of $\widehat{G_r}$) containing $u$, by conjugation. 
\end{lemm}
\begin{proof}
This follows from direct computations.
\end{proof}

Now identify identify the flag variety $\mathcal{B}$ of $\widehat{G_r}$ to be the quotient of $\widehat{G_r}$ by the standard upper triangular Borel subgroup of $\widehat{G_r}$, and identify $S_{nr}$ with the Weyl group defined by the diagonal maximal torus inside this Borel. So we can talk about the Borel subgroups at relative positions indexed by $S_{nr}$ and view them as elements in $\mathcal{B}$. This indicates us a natural way to construct a variety out of $\mathcal{B}_u$, on which $G_r^F$, instead of the whole connected group $G_r$, acts:
\begin{defi}\label{defi:B_{u,w}}
Let $\mathcal{B}_{u,w}:=\mathcal{B}_u\cap X_w$, the intersection (in the flag variety $\mathcal{B}$) of $\mathcal{B}_u$ and $X_w$, where $X_w$ denotes the Deligne--Lusztig variety consisting of the Borel subgroups lying on the relative position $w\in S_{nr}$ with their Frobenius image. 
\end{defi}
For a fixed $r\in\mathbb{Z}_{>1}$, a representation $R$ of $G_r^F$ is called primitive if $R$ does not factor through the lower level group $G_{r-1}^F$. It would be interesting to figure out for which $w$ the virtual representation 
$$R_{u,w}:=\sum_i(-1)^iH_c^i(\mathcal{B}_{u,w},\overline{\mathbb{Q}}_{\ell})$$ 
of $G_r^F$ is primitive. 

\vspace{2mm} Indeed, it can happen that $\mathcal{B}_{u,w}$ is actually empty; even if it is non-empty it can happen that $R_{u,w}$ is non-primitive. We determine these situations for the single cycles $(1,...,z)$:

\begin{thm}\label{thm:R_{u,w}}
Let $w=(1,...,z)\in S_{nr}$ with $z\in[1,nr]$. Then
\begin{itemize}
\item[(i.)] If $z=1$, then $R_{u,w}$ is a primitive permutation representation.
\item[(ii.)] If $1< z<n$, then $R_{u,w}$ is primitive.
\item[(iii.)] If $z=n$, then $R_{u,w}\neq0$ is non-primitive.
\item[(iv.)] If $n< z\leq nr$, then $\mathcal{B}_{u,w}=\emptyset$ (so $R_{u,w}=0$).
\end{itemize}
\end{thm}
\begin{proof}
In the below we would interpret a point in $\mathcal{B}$ as a complete flag; the translation of relative positions in terms of flags can be found in \cite[I.4.3]{AG4} (see also \cite{Webster_RelativePosition}). To make the argument coherent, we would treat (i.) first, then (iv.), (iii.), and finally (ii.).

\vspace{2mm}
From now on $\widehat{G_r}$ is understood as the automorphism group of the $\overline{\mathbb{F}}_q$-space $V$ with the standard basis 
\begin{equation}\label{formula:basis}
\{ x_1^{(0)},x_2^{(0)},...,x_n^{(0)},x_1^{(1)},x_2^{(1)},...,x_n^{(r-1)} \}.
\end{equation}
We use $\mathcal{F}\colon V_1\subseteq V_2\subseteq...\subseteq V_{nr}=V$ with $\dim V_i=i$ as a general notation for a complete flag of $V$.

\vspace{2mm} (i.) In this case $\mathcal{B}_{u,w}$ is the finite set of $\mathbb{F}_q$-rational flags of $V$ stabilised by $u$. In particular $R_{u,w}=\overline{\mathbb{Q}}_{\ell}[\mathcal{B}_{u,w}]$ is a permutation representation. Consider the complete flag $\mathcal{F}\colon V_1\subseteq V_2\subseteq...\subseteq V_{nr}=V$ given by
\begin{equation*}
\begin{split}
V_1
& =\langle x_n^{(r-1)} \rangle \subseteq \langle x_n^{(r-1)}, x_{n-1}^{(r-1)} \rangle\subseteq...\subseteq \langle  x_n^{(r-1)}, x_{n-1}^{(r-1)},...,x_2^{(r-1)} \rangle\\
& \subseteq \langle x_n^{(r-1)}, x_{n-1}^{(r-1)},...,x_2^{(r-1)}, x_n^{(r-2)} \rangle \subseteq\langle x_n^{(r-1)}, x_{n-1}^{(r-1)},...,x_2^{(r-1)}, x_n^{(r-2)}, x_{n-1}^{(r-2)} \rangle\\
& \subseteq...\subseteq  \langle x_n^{(r-1)},...,x_2^{(r-1)}, x_n^{(r-2)},...,x_2^{(r-2)}  \rangle   \\
&\subseteq ...\subseteq \langle x_n^{(r-1)},...,x_2^{(r-1)},..., x_n^{(1)},...,x_2^{(1)} \rangle =V_{(n-1)(r-1)}  \\
& \subseteq \langle V_{(n-1)(r-1)}, x_n^{(0)}+x_1^{(r-1)} \rangle=V_{(n-1)(r-1)+1} \\
& \subseteq \langle V_{(n-1)(r-1)+1}, x_1^{(r-1)} \rangle \subseteq \langle V_{(n-1)(r-1)+1}, x_1^{(r-1)}, x_1^{(r-2)} \rangle \\
& \subseteq... \subseteq \langle V_{(n-1)(r-1)+1}, x_1^{(r-1)},...,x_1^{(1)} \rangle=V_{nr-n+1} \subseteq \langle V_{nr-n+1}, x_{n-1}^{(0)} \rangle \\
& \subseteq \langle V_{nr-n+1}, x_{n-1}^{(0)}, x_{n-2}^{(0)} \rangle \subseteq...\subseteq \langle V_{nr-n+1}, x_{n-1}^{(0)},..., x_{1}^{(0)} \rangle=V_{nr}.
\end{split}
\end{equation*}
Note that $\mathcal{F}\in \mathcal{B}_{u,w}$. Also note that, by construction every vector in the piece $V_{(n-1)(r-1)+1}$ has the property that its $n$-th coordinate equals to its $(n(r-1)+1)$-th coordinate. However, this is a property not preserved under the action of $(G^{r-1}_r)^F$. So $\mathcal{F}$ is an element of $\mathcal{B}_{u,w}$ not stabilised by $(G^{r-1}_r)^F$, which implies that $R_{u,w}$ is primitive.

\vspace{2mm}
(iv.) In this case $X_{w}$ is the moduli of complete flags $\mathcal{F}$ whose first $z$ pieces are of the form (see also \cite[2.2]{DL1976})
$$V_1\subseteq V_1+FV_1\subseteq V_1+FV_1+F^2V_1\subseteq \cdots \subseteq ... \subseteq \bigoplus_{j=0}^{z-1}F^jV_1=V_z\cong\overline{\mathbb{F}}_q^{z},$$
where $\dim V_1=1$, and $V_z$ together with all the other pieces are $\mathbb{F}_q$-rational. If this flag is contained in $\mathcal{B}_{u,w}$, then $V_1=\langle v_1 \rangle$ is stabilised by $u$, so we have $uv_1=v_1$ as $u$ is unipotent, which implies that $v_1$ is a vector of the form $(0,...,0,a_1,...,a_n)^T$. Note that the orbit of such a vector under $F$ can only generate a space of dimension at most $n$, smaller than $z$. Thus $\mathcal{B}_{u,w}=\emptyset$. 

\vspace{2mm}
(iii.) In this case, similar to the above discussion we have that, any flag $\mathcal{F}$ in $\mathcal{B}_{u,w}$ must be of the following form: (a) The first $n$ pieces are
$$V_1\subseteq V_1+FV_1\subseteq V_1+FV_1+F^2V_1\subseteq \cdots \subseteq ... \subseteq \bigoplus_{j=0}^{n-1}F^jV_1=V_n\cong\overline{\mathbb{F}}_q^{n},$$
where $V_1$ is generated by a vector $v_1$ whose all but the last $n$ coordinates are zero, and (b) $V_n$ and other pieces of $\mathcal{F}$ are $F$-stable. In particular, the $n$-th piece is a fixed space (see the notation of \eqref{formula:basis})
$$V_n=\langle x_1^{(r-1)},x_2^{(r-1)},...,x_n^{(r-1)} \rangle.$$
Now, to form a $V_{n+1}$ (so that the flag lies in $X_{w}$), it suffices to choose an $F$-stable line in the space generated by $\{ x_1^{(0)},x_2^{(0)},...,x_n^{(0)},x_1^{(1)},x_2^{(1)},...,x_n^{(r-2)} \}$; by identifying a line in the vector space as a point in the projective space, this means to choose an $\mathbb{F}_q$-point of $\mathbb{P}^{n(r-1)-1}_{ \overline{\mathbb{F}}_q }$. Continuing this process we see that $\mathcal{B}_{u,w}$ is a non-empty union of irreducible components of $X_w$, and thus isomorphic to a union of components of 
$$Y_z\times \prod_{j=1}^{n(r-1)-1} \mathbb{P}^{j}(\mathbb{F}_q),$$
where $Y_z$ is the Coxeter variety, i.e.\ the connected variety of all $v\in \mathbb{P}^{z-1}_{\overline{\mathbb{F}}_q}$ which do not lie on any $\mathbb{F}_q$-rational hyperplane. Note that by construction the subgroup $(G_r^{r-1})^F$ (given by the ``$A_{r-1}$-part'' in the image of the group embedding) does not permute these irreducible components, and in each component the action of $(G_r^{r-1})^F$ is trivial because it fixes every vector of $V_n$. So the subgroup $(G_r^{r-1})^F$ acts on $\mathcal{B}_{u,w}$ trivially, and thus the representations $H_c^i(\mathcal{B}_{u,w},\overline{\mathbb{Q}}_{\ell})$ of $G_r^F$ are all non-primitive. Since the alternating sum of cohomology groups of $Y_z$ is well-known to be non-zero, the vector space $R_{u,w}$ is non-zero, as desired.

\vspace{2mm}
(ii.) Similar to (iii.), in this case $B_{u,w}$ is a disjoint union of copies of $Y_z$, labelled by some $\mathbb{F}_q$-rational partial flags 
$$V_z\subseteq V_{z+1}\subseteq...\subseteq V_{nr}$$ 
with $V_z\subseteq\langle x_1^{(r-1)},...,x_n^{(r-1)} \rangle$, and $(G^{r-1})^F$ fixes the vectors in $V_z$. Thus $R_{u,w}$ is actually a permutation representation of $(G^{r-1})^F$. So, as in the situation of (i.), it suffices to construct a partial flag of the above type on which $(G^{r-1})^F$ acts non-trivially: We can just truncate the flag constructed in (i.), because $z\leq n-1 < (n-1)(r-1)+1$.
\end{proof}

\section{$\mathrm{SL}_2$ over finite dual numbers}\label{sec:expl}

In this section we assume that $\mathbb{G}=\mathrm{SL}_2$, $r=2$, and $p=\mathrm{char}(\mathbb{F}_q)$ odd, and denote by $\varepsilon$ the image of $\pi$ in $\mathcal{O}_r$ (so $\varepsilon^2=0$). Note that $G\cong {\mathrm{SL}_2}_{/\overline{\mathbb{F}}_q}\ltimes \mathrm{Lie}(\mathrm{SL}_2)_{/\overline{\mathbb{F}}_q}$ via the adjoint action. 

\vspace{2mm} The only possible degrees of irreducible characters of $G^F$ are (see e.g.\ \cite[Section~3]{Lusztig2004RepsFinRings}): $$1,q,q+1,\frac{q+1}{2},q-1,\frac{q-1}{2},q^2+q,q^2-q,\frac{q^2-1}{2}.$$
Here we only concern the last three ones, as the others are for non-primitive characters (i.e.\ the characters factor through the lower level group $G_1^F=\mathrm{SL}_2(\mathbb{F}_q)$). It is known that, for $\mathrm{SL}_2$, the orbits of primitive characters are always regular and their adjoint centralisers are abelian (see \cite[3.1]{Sta2011ExtendedDL}). In the below we describe how the twisting operator $\mathrm{Sh}_G$ affect these primitive characters; most properties of these characters that we need can be found in \cite{Lusztig2004RepsFinRings} and \cite{Sta2011ExtendedDL}, and we would omit the details when this is the case.

\vspace{2mm} From the work of Lusztig \cite{Lusztig2004RepsFinRings} we know that the degrees $q^2+q$ and $q^2-q$ are those afforded by $\mathfrak{R}_{T,U}^{\theta}$ ($:=\mathfrak{R}_{T,U,r}^{\theta}$) with the $\theta$'s satisfying certain generic condition; the orbits of these characters are regular semisimple. From the work of Stasinski \cite{Sta2011ExtendedDL} we know that most of the characters of degree $\frac{q^2-1}{2}$ are missed by any $\mathfrak{R}_{T,U}^{\theta}$; the orbits of these characters are regular nilpotent. 

\vspace{2mm} {Case-I: $\deg=q^2\pm q$.} We want to show that the irreducible characters of these degrees are all invariant under $\mathrm{Sh}_G$. The centraliser of the semisimple part of an element in $G_1={\mathrm{SL}_2}_{/\overline{\mathbb{F}}_q}$ is either a torus or the whole group, so by Theorem~\ref{thm:main result}, Proposition~\ref{prop:unipotent centraliser prop}, and the above discussions, it suffices to show that, if 
$A=\begin{bmatrix}
1 & u   \\
0 & 1
\end{bmatrix}
\cdot
\begin{bmatrix}
1+a\varepsilon & b\varepsilon   \\
c\varepsilon & 1-a\varepsilon
\end{bmatrix}$
with $u,a,b,c\in\overline{\mathbb{F}}_q$ and $u\neq0$, then $A\in Z_{G}(A)^{\circ}$. It is enough to show that, the subgroup (of $G$) consisting of the elements of the form 
$X=\begin{bmatrix}
1 & \upsilon   \\
0 & 1
\end{bmatrix}
\cdot
\begin{bmatrix}
1+x\varepsilon & y\varepsilon   \\
z\varepsilon & 1-x\varepsilon
\end{bmatrix}$ ($\upsilon,x,y,z\in\overline{\mathbb{F}}_q$), that commuting with $A$,
is connected. Indeed, simplifying the equation
$$XA=AX$$
we get
$$z=\frac{c}{u}\upsilon,\ x=\frac{-c}{2u}\upsilon^2+\frac{cu+2a}{2u}\upsilon,$$
so the space of solutions is isomorphic to an affine space, thus connected.

\vspace{2mm} {Case-II: $\deg=\frac{q^2-1}{2}$.} As we mentioned before, these are nilpotent orbit characters. There are two non-trivial nilpotent orbits, with representatives of the form
$o_1=\begin{bmatrix}
0 & \square\in \mathbb{F}_q^{\times} \\
0 & 0
\end{bmatrix}$
and $o_2:=\begin{bmatrix}
0 & \slashed{\square}\in \mathbb{F}_q^{\times} \\
0 & 0
\end{bmatrix}$, respectively, in $\mathfrak{g}^F\cong (G^1)^F$. Note that the centralisers of such elements (in $G^F$) are  
$$Z_i:=\left(\{ \pm1 \}\times 
\begin{bmatrix}
1 & \mathbb{F}_q \\
0 & 1
\end{bmatrix}\right)\ltimes (G^1)^F\subseteq G^F.$$
Let $\psi_1$ (resp.\ $\psi_2$) be an extension (to $Z_i$) of a fixed irreducible character of $(G^1)^F$ corresponding to $o_1$ (resp.\ $o_2$), then according to Clifford theory (see e.g.\ \cite[6.17]{Isaacs_CharThy_Book}), the nilpotent primitive characters of $\mathrm{SL}_2(\mathbb{F}_q[\varepsilon])$ are precisely
$$R_{\chi,i}:=\mathrm{Ind}_{Z_i}^{G^F}\chi \psi_i\quad (i=1,2),$$
where $\chi$ runs over the irreducible characters of the left factor of $Z_i$. In particular, there are totally $2\times q+2\times q=4q$ such characters, as been indicated in \cite[3.1]{Lusztig2004RepsFinRings}. Now pick an element $X'=\begin{bmatrix}
-1 & x'\varepsilon\neq0 \\
0 & -1
\end{bmatrix}$. Note that the action of $n_F$ multiplies $x'\varepsilon$ in this matrix by a non-square element in $\mathbb{F}_q$. Since $\psi_i$ is non-trivial on exactly one of the orbits of $o_i$, we see that, via the formula of induced characters, the two class functions $\mathrm{Sh}_G(R_{\chi,i})$ and $R_{\chi,i}$ are always different at $X'$.

\vspace{2mm} Summarising Case-I and Case-II we get:

\begin{coro}\label{coro:semisimple of SL_2 are twist-invariant}
Let $R$ be a primitive irreducible character of $\mathrm{SL}_2(\mathbb{F}_q[\varepsilon])$, then $R$ is a semisimple orbit character if and only if $\mathrm{Sh}_G (R)=R$.
\end{coro}

We expect that this criterion can be extended to $\mathrm{SL}_n(\mathbb{F}_q[[\pi]]/\pi^r)$ for any $n,r\geq 2$, in good characteristics.

\bibliographystyle{alpha}
\bibliography{zchenrefs}

\end{document}